\documentclass[a4paper, 11pt]{article}
\usepackage{amsmath, amssymb, amsthm, latexsym}
\usepackage{amsrefs}
\usepackage{xspace}
\usepackage{a4wide}
\usepackage{authblk}
\usepackage[usenames, dvipsnames]{color}

\title{Orthogonal forms of Kac--Moody groups \\ are acylindrically hyperbolic}

\author[1]{Pierre-Emmanuel Caprace\thanks{F.R.S.-FNRS research associate, supported in part by the ERC (grant \#278469)}}

\author[1]{David Hume\thanks{Supported by the ERC (grant \#278469)}}

\affil[1]{IRMP, UCLouvain, 1348 Louvain-la-Neuve, Belgium}

\date{First draft: October 13, 2014; revised: February 25, 2015}

\newtheorem*{btheorem}{Main Theorem}
\newtheorem*{bcorollary}{Corollary}
\newtheorem{proposition}{Proposition}[section]
\newtheorem{theorem}[proposition]{Theorem}
\newtheorem{lemma}[proposition]{Lemma}
\newtheorem{corollary}[proposition]{Corollary}

\theoremstyle{definition}
\newtheorem{definition}[proposition]{Definition}
\newtheorem{remark}[proposition]{Remark}

\newcommand{\fgen}[1]{\left\langle #1 \right\rangle}

\newcommand{\Stab}{\textup{Stab}}

\newcommand{\set}[1]{\left\{#1\right\}}
\newcommand{\setcon}[2]{\left\{#1\ \left|\ #2\right.\right\}}

\newcommand{\CC}{\mathbf{C}}
\newcommand{\R}{\mathbf{R}}
\newcommand{\Z}{\mathbf{Z}}

\newcommand{\tu}{\textup}

\newcommand{\vareps}{\varepsilon}
\newcommand{\la}{\langle}
\newcommand{\ra}{\rangle}
\newcommand{\Aut}{\mathrm{Aut}}
\newcommand{\Isom}{\mathrm{Is}}
\newcommand{\cat}{$\mathrm{CAT}(0)$\xspace}
\newcommand{\ad}{\textup{ad}\ }

\begin{document}

\maketitle

\begin{abstract}
We give sufficient conditions for a group acting on a geodesic metric space to be acylindrically hyperbolic and mention various applications to groups acting on CAT($0$) spaces. We prove that a group acting on an irreducible non-spherical non-affine building is acylindrically hyperbolic provided there is a chamber with finite stabiliser whose orbit contains an apartment. Finally, we show that the following classes of groups admit an action on a building with those properties: orthogonal forms of  Kac--Moody groups over arbitrary fields, and irreducible graph products of arbitrary groups - recovering a result of Minasyan--Osin. 
\end{abstract}

\section{Introduction}

An \textbf{acylindrically hyperbolic} group, defined by Osin \cite{Os13}, is a group $G$ which admits a non-elementary acylindrical action on a hyperbolic space.

This notion is a far-reaching generalisation of Gromov's relative hyperbolicity. The class of groups satisfying this includes many widely-studied families of finitely generated groups: mapping class groups, outer automorphism groups of free groups, groups acting properly and essentially  on irreducible finite-dimensional \cat cube complexes and small cancellation groups, see \cite{Os13} and references therein.

However,  a remarkable feature of the definition of acylindrical hyperbolicity is that it also allows for examples which are not finitely generated. It follows from results of Cantat--Lamy \cite{CantLamy-13-Cremacylhyp} that the Cremona group (i.e. the group of bi-rational transformations of the complex projective plane) is acylindrically hyperbolic, using an action of this group on an infinite dimensional hyperbolic space.  This provides the first example of a (non-locally compact) connected Hausdorff topological group which is acylindrically hyperbolic.

One goal of this paper is to provide more examples of connected topological groups which are acylindrically hyperbolic. Our examples are orthogonal forms of Kac--Moody groups over arbitrary fields; by definition, Kac--Moody groups over fields are taken in the sense of  Tits \cite{Tits-87-Functor}, and their orthogonal forms are defined as the centraliser of the Chevalley involution (cf. Definition $\ref{defn:orthform}$).

\begin{btheorem} 
Let $A$ be a generalized Cartan matrix of irreducible, non-spherical, non-affine type, and $\mathcal G_A(F)$ be a Kac--Moody group of  type $A$ over a field $F$.  

Then the orthogonal form $\mathcal K_A(F)$ is acylindrically hyperbolic.
\end{btheorem}

If we set $F$ to be the field of order~$2$, then $\mathcal K_A(F)$ is isomorphic to the Weyl group $W$, and the property of acylindrical hyperbolicity is not new: indeed, all irreducible, infinite, non-affine Coxeter groups act properly cocompactly on a \cat space with rank~$1$ elements (see \cite[Cor.~4.7]{CapFuj-10-rnk1}), and this is a sufficient condition for acylindrical hyperbolicity (see \cite{Sisto-contracting} or Proposition~\ref{prop:Sisto} below).  However,  the main motivating situation is when the ground field $F$ is the field of real numbers. In that case  the orthogonal form $\mathcal K_A(\R)$ is a factor in the Iwasawa decomposition of $\mathcal G_A(\R)$ and is a connected Hausdorff topological group when endowed with the Kac--Peterson topology (see \cite{HartnickKohlMars} and references therein). We remark that, over $\R$, the orthogonal forms have only recently been shown to be not simple. This is in stark contrast to the situation for real simple Lie groups, where the centraliser of a Chevalley involution is a compact group, and is simple in many cases \cite[pp.~520--534]{Helg-DG}. Non-trivial quotients were first constructed on the Lie algebra level by physicists for certain types \cites{DBHP06,DKN06}; their arguments were later extended, and promoted to the group level, in \cites{HKL11,GHKW14}. The quotient maps $\Xi \colon \mathcal K_A(\R) \to Q$ contructed in this way are called \textbf{maximal spin representations} of the orthogonal form $\mathcal K_A(\R)$. The quotient group $Q$ happens to be a (finite-dimensional) compact Lie group. Our result on acylindrical hyperbolicity provides many more quotients of those orthogonal forms $\mathcal K_A(\R)$, of a very different flavour: indeed, it is proved in \cite[Th.~8.5]{DGO11}  that all acylindrically hyperbolic groups have non-trivial proper free normal subgroups, and are even \textbf{SQ-universal}, which means that every countable group embeds in one of their quotients. 

Moreover, the criteria for acylindrical hyperbolicity that will be established below applies not only to the orthogonal form $\mathcal K_A(\R)$, but also to many of its subgroups. In particular it will follow that the kernel $\mathrm{Ker}(\Xi)$ of the maximal spin representation is acylindrically hyperbolic (hence SQ-universal), at least for simply laced diagrams, see Corollary~\ref{cor:KerSpinRep}.

\medskip
The Main Theorem will follow from a general criterion of acylindrical hyperbolicity for groups acting on buildings, described below, which
also yields the following result on graph products (see Section~\ref{sec:GraphProd} for the definitions). This was first proved by Minasyan--Osin \cite[Corollary $2.13$]{minas-osin-13} using a different method. 

\begin{bcorollary}\label{cor:GraphProd}
An irreducible graph product of arbitrary groups is acylindrically hyperbolic or virtually cyclic. 
\end{bcorollary}

\subsection{Intermediate results}

The Main Theorem and the Corollary above  are deduced from  the three following results, which should be of independent interest.

The definition of an acylindrically hyperbolic group requires an action of that group on some hyperbolic space. In our setting, it is more natural to consider the actions of Kac--Moody groups on buildings, which are only hyperbolic when the underlying Weyl group is hyperbolic. In order to cover group actions on buildings of arbitrary  non-spherical and non-affine type, we use the projection complexes introduced by Bestvina--Bromberg--Fujiwara which build actions of groups on quasi-trees \cite{BBF13}. In doing so, we obtain sufficient conditions for the existence of non-trivial hyperbolically embedded subgroups from the action of a group on a wider range of geodesic metric spaces. The following statement was implicitly contained in the work of Bestvina--Bromberg--Fujiwara \cite{BBF13} and now appears explicitly as Theorem H in \cite{BBF10}.

\begin{theorem}[Bestvina--Bromberg--Fujiwara]\label{WPDstrcontAcylhyp}
Let $G$ be a group which acts by isometries on a geodesic metric space $X$. Suppose $h\in G$ is of infinite order, has positive translation length and satisfies the following conditions.
\begin{enumerate}
\item The action of $h$ on $X$ is \textbf{weakly properly discontinuous} (or \textbf{WPD} for short): for every $D>0$ and $x\in X$ there exists some $m>0$ such that
\[
 \setcon{g\in G}{d(x,gx)<D \tu{ and } d(h^mx,gh^mx)<D}
\]
is finite.
\item The $X$-orbit of $H=\fgen{h}$ is \textbf{strongly contracting}: there exists a point $x\in X$ with associated closest point projection $\pi_x:X\to H(x)$ such that 
\[
 \sup\setcon{\tu{diam}(\pi_x(B(y;r)))}{B(y;r)\cap H(x)=\varnothing}
\]
is finite.
\end{enumerate}
Then there is a subgroup $H'$ containing $H$ as a finite index subgroup which is hyperbolically embedded in $G$. Moreover, $G$ is either virtually cyclic or acylindrically hyperbolic.
\end{theorem}

Notice that no assumption is made on the group $G$, just on the action. The strongly contracting property asks that $X$ has, in some sense, a `hyperbolic direction'.

When the first version of this paper appeared, Theorem $\ref{WPDstrcontAcylhyp}$ did not appear in the existing versions of \cite{BBF10} or \cite{BBF13} so a proof was given in this paper. It remains here for completeness.

\medskip
In Section~\ref{sect:suffacylhyp} below, we present various applications to Theorem~\ref{WPDstrcontAcylhyp} for groups acting properly on \cat spaces and containing \textbf{rank~$1$ isometries} (i.e. hyperbolic  isometries whose axes do not bound a half-flat). We show that a \cat group contains rank~$1$ isometries if and only if it is acylindrically hyperbolic or virtually cyclic (see Corollary~\ref{cor:cat0groups}). Moreover, any (possibly non-uniform) lattice in the isometry group of a proper cocompact \cat space is acylindrically hyperbolic as soon as the ambient space admits rank~$1$ isometries (Corollary~\ref{cor:Lattices}). We also point out an application to groups acting properly on (possibly non-proper) finite-dimensional \cat cube complexes (Corollary~\ref{cor:CCC}). Most of those applications seems to be new. They imply moreover  that certain lattices appearing in the realm of Kac--Moody groups over \emph{finite} fields are acylindrically hyperbolic (Remark~\ref{rem:KM}). It should however be emphasized that, in all those applications, the WPD hypothesis required by Theorem~\ref{WPDstrcontAcylhyp} is guaranteed by the fact that the action of the ambient group under consideration is metrically proper. 

However, the Main Theorem typically involves groups whose action on the associated building is not proper. Indeed,  for Kac--Moody groups over the reals, the action of the orthogonal form $\mathcal K(\R)$ on its building is \emph{not} metrically proper. In fact the stabiliser in $\mathcal K(\R)$ of any singular point, i.e. any point which is not in the interior of a chamber, is an uncountable compact subgroup. The key property of those actions, which we shall exploit and which will turn out to be sufficient to prove acylindrical hyperbolicity, is that the stabiliser of every \emph{regular} point (i.e. a point lying in the interior of a chamber) happens to be finite. 
 
\medskip
Those considerations naturally lead us to the second part, where we find conditions on a (not necessarily proper) group action on a building which ensure that the group has a WPD element. Let $X$ be a building of type $(W,S)$. Throughout, the generating set $S$ is assumed finite. An automorphism $h\in\Aut(X)$ is said to be \textbf{regular} if it is hyperbolic and if it has an axis $\ell$ contained in an apartment $\mathcal{A}$ such that $\ell$ is not contained in a bounded neighbourhood of any wall of $\mathcal{A}$. One shows that, in that case, every $h$-axis is contained in the apartment $\mathcal A$, see Lemma~\ref{lem:bi-infinite} below. Therefore, it makes sense to define the \textbf{combinatorial hull} of $h$,  as the combinatorial convex hull of $\ell$, i.e. the smallest gallery-convex set of chambers containing $\ell$. If the building is thick, this coincides with the intersection of all apartments containing $\ell$. One verifies that this is indeed independent of the choice of $\ell$ (see Lemma~\ref{lem:bi-infinite} below). The shape of the combinatorial hull can be very different depending on the type of the building. Indeed, if the building is Euclidean, the  combinatorial hull of $h$ is a whole apartment (hence, a maximal flat of $X$). If $X$ is thick, non-Euclidean and $h$ is a rank 1 isometry, then the combinatorial hull is a bounded neighbourhood of $\ell$. 

\begin{theorem}\label{regularWPD} 
Let $G$ be a group acting on a building $X$. Let $h \in G$ be a regular element whose combinatorial hull contains a chamber with finite stabiliser in $G$.  Then the action of $h$ on $X$ is WPD. 
\end{theorem}

By assumption, the building $X$ must be non-spherical. The hypothesis on chamber stabilisers is critical for the WPD condition. 

\medskip
In the  third part, we combine Theorems~\ref{WPDstrcontAcylhyp} and~\ref{regularWPD} in order to establish  the following criterion of acylindrical hyperbolicity for groups acting on buildings. 

\begin{theorem}\label{thm:AcylBuilding}
Let $X$ be a building of irreducible non-spherical, non-affine type $(W, S)$, and let $G \leq \Aut(X)$. Assume that  there exists a chamber  with finite stabiliser, and whose $G$-orbit contains a subset of an apartment which is an orbit under a finite index subgroup of the Weyl group $W$ (e.g. the $G$-orbit of that chamber contains an apartment). 

Then $G$ is acylindrically hyperbolic.
\end{theorem}

The Main Theorem is deduced from Theorem~\ref{thm:AcylBuilding} by showing that orthogonal forms of Kac--Moody groups satisfy its hypotheses, see Proposition~\ref{findeltOrthform} below. The key point is that the orthogonal form $\mathcal K_A(F)$ has a natural action with finite chamber stabilisers on a building $X$, which contains moreover an apartment $\mathcal A$ such that the stabiliser $\tilde W = \Stab_{\mathcal K_A(F)}(\mathcal A)$ is transitive on the chambers of $\mathcal A$. The group $\tilde W$ is called the \textbf{extended Weyl group}, and is isomorphic to an extension of the Weyl group $W$ by an elementary abelian group of order $2^n$, where $n$ is the size of the generalized Cartan matrix $A$. Therefore, Theorem~\ref{thm:AcylBuilding} implies the  following sharper version of the Main Theorem.

\begin{corollary}\label{cor:sharper}
Let $\mathcal K_A(F)$ be as in the Main Theorem. Any subgroup $H \leq \mathcal K_A(F)$ containing a finite index subgroup of the extended Weyl group $\tilde W$ is acylindrically hyperbolic (hence non-simple).  
\end{corollary}

In their recent work \cite{GHKW14}, Ghatei--Horn--K\"{o}hl--Wei{\ss} have shown that in case $A$ is simply laced (i.e. all off-diagonal entries belong to $\{0, -1\}$), the image of the extended Weyl group under the maximal spin representation $\Xi \colon \mathcal K_A(\R) \to Q$ is finite. In other words, this means that $\mathrm{Ker}(\Xi)$ contains a finite index subgroup of the extended Weyl group $\tilde W$. The previous corollary therefore implies the following. 

\begin{corollary}\label{cor:KerSpinRep}
If $A$ is simply laced, then the kernel of the maximal spin representation $\Xi \colon \mathcal K_A(\R) \to Q$ is acylindrically hyperbolic (hence non-simple). \qed
\end{corollary}

\subsection*{Plan of the paper}
We introduce orthogonal forms in section \ref{sect:orthogforms}, recalling the fundamentals of Kac--Moody algebras and groups, twin buildings and BN pairs as part of this. Sections \ref{sect:suffacylhyp}, \ref{sect:regisoms} and \ref{sec:AcylBuilding} deal with the proofs of Theorems \ref{WPDstrcontAcylhyp}, \ref{regularWPD} and \ref{thm:AcylBuilding}   respectively. Section~\ref{sect:suffacylhyp} also contains applications showing that various groups acting properly on \cat spaces are acylindrically hyperbolic. Those examples include several families of lattices in locally finite buildings constructed from Kac--Moody groups over finite fields.  In the  final section~\ref{sect:othformaction} we come back to orthogonal forms of Kac--Moody groups and show that they satisfy the hypotheses of Theorem~\ref{thm:AcylBuilding}. We also point out that graph products of arbitrary groups fall within the scope of Theorem~\ref{thm:AcylBuilding}, thereby deducing the Corollary stated above.

\subsection*{Acknowledgements}
The authors would like to thank Ralf K\"{o}hl, Denis Osin and an anonymous referee for comments on earlier versions of this paper, and Timoth\'ee Marquis for useful remarks on straight elements.

\section{Orthogonal forms}\label{sect:orthogforms}
Let $A=(A_{ij})$, $i,j\in\set{1,\dots,n}$, be a \textbf{generalised Cartan matrix}: an $n\times n$ matrix where $A_{ii}=2$ for each $i$, $A_{ij}$ is a non-negative integer whenever $i\neq j$ and $A_{ij}=0$ if and only if $A_{ji}=0$.

We define a \textbf{Kac--Moody algebra} $\mathfrak{g}_A$ as the Lie algebra over $\CC$ (or $\R$) generated by $3n$ elements $e_i,f_i,h_i$ satisfying the relations
\begin{itemize}
 \item $[h_i,h_j]=0$ and $[e_i,f_j]=\delta_{ij}h_i$,
 \item $[h_i,e_j] = A_{ij}e_j$ and $[h_i,f_j]=-A_{ij}f_j$,
 \item $(\ad e_i)^{1-A_{ij}} e_j = 0$ and $(\ad f_i)^{1-A_{ij}} f_j = 0$, whenever $i\ne j$.
\end{itemize}
Each triple $\set{e_i,f_i,h_i}$ generates a subalgebra $\mathfrak{g}_i$ of $\mathfrak{g}_A$, which is isomorphic to $\mathfrak{sl}_2$.
The $h_i$ generate an abelian subalgebra, called the \textbf{Cartan subalgebra}, which we denote by $\mathfrak{h}$.

The algebra $\mathfrak g_A$ admits an involutory automorphism called the \textbf{Chevalley involution} $\omega$, defined by
\[
 \omega(e_i)=-f_i,\quad \omega(f_i)=-e_i, \quad \omega(h_i)=-h_i.
\]
The subalgebra of elements fixed by $\omega$ is denoted by $\mathfrak{k}_A$. This is often referred to as the \textbf{maximal compact subalgebra}.

Now let $F$ be an arbitrary field. The (simply connected) \textbf{Tits' functor}, defined by Tits~\cite{Tits-87-Functor}, associates a unique  Kac--Moody group $\mathcal G_A(F)$ to the Lie algebra $\mathfrak{g}_A$, where each $\mathfrak{g}_i$ is associated to a subgroup $G_i$ of $\mathcal G_A(F)$ isomorphic to $\mathrm{SL}_2(F)$ and $\mathfrak{h}$ is associated to an abelian subgroup $T\cong(F^\times)^n$ . In some of the literature Kac--Moody groups which arise in this way are called \emph{split minimal}. We briefly recall some properties of that construction. 

There are injective homomorphisms $\psi_i \colon \mathrm{SL}_2(F)\to \mathcal G_A(F)$ and $\eta \colon (F^\times)^n\to \mathcal G_A(F)$ with images $G_i$ and $T$ respectively. The diagonal matrices in each $G_i$ are identified with elements of $T$ by the relations
\[
 \psi_i
  \left(
   \begin{array}{cc}
   y & 0 \\ 0 & y^{-1}
   \end{array}
  \right)
 = \eta\big(y(\delta_{1i},\dots, \delta_{ni})\big).
\]
Moreover the group $\mathcal G_A(F)$ is generated by $T \cup (\bigcup_{i=1}^n G_i)$. We refer the reader to \cite{Tits-87-Functor} for a full description of the defining relations of $\mathcal G_A(F)$ with respect to that generating set. 

A Kac--Moody group $\mathcal G_A(F)$ admits a (saturated) \textbf{twin BN pair} $(B^+,B^-,N)$, i.e.  a triple of subgroups   satisfying the following conditions:
\begin{itemize}
 \item $G=\fgen{B^+,N}=\fgen{B^-,N}$,
 \item $T=B^+\cap B^- = B^+\cap N = B^- \cap N\cong  (F^\times)^n$ is a normal subgroup of $N$,
 \item $W=N/T$ is a rank $n$ Coxeter group generated by a set of reflections $S$,
 \item for every $w\in W$, $s\in S$ and $\vareps\in\set{+,-}$, we have $sB^\vareps w\subseteq B^\vareps wB^\vareps\cup B^\vareps swB^\vareps$,
 \item for every $s\in S$ and $\vareps\in\set{+,-}$, we have $sB^\vareps s\not\subseteq B^\vareps$.
\end{itemize}

Using the notation above, the subgroup $B^+$ (resp. $B^{-}$) is generated by $T$ and all elements of the form
\[
\psi_i
   \left(
    \begin{array}{cc}
    1 & y \\ 0 & 1
    \end{array}
   \right)
\quad
\tu{resp.}\ \ 
\psi_i
   \left(
    \begin{array}{cc}
    1 & 0 \\ y & 1
    \end{array}
   \right)
\]
while $N$ is generated by $T$ and all elements of the form
\[
 r_i = \psi_i
    \left(
     \begin{array}{cc}
     0 & 1 \\ -1 & 0
     \end{array}
    \right),
\]
which represent elements of the generating set $S$ in $W = N/T$. 

The coset spaces $G/B^+$ (resp. $G/B^-$) can be viewed as the set of chambers of a building $X^+$ (resp. $X^-$) of type $(W,S)$. Letting $G$ act diagonally by left translation on $X^+\times X^-$ we recover $B^\vareps$ as the stabiliser of the chamber $C^\vareps$ associated to the trivial coset $B^\vareps$ and $T$ is the stabiliser of the \textbf{standard twin chamber} $(C^+,C^-)$. The two buildings $X^+$ and $X^-$ are related by a so-called twinning; the twin chamber $(C^+,C^-)$ is contained in a unique twin apartment $\mathcal A$, called the \textbf{standard twin apartment}. We then recover $N$ as the stabiliser of $\mathcal{A}$. The elements $r_i$ defined above stabilise $\mathcal A$ and act as the reflections across the walls of the standard chambers $C^+$ and $C^-$.

Via the Tits' functor, we may also define a Chevalley involution on a Kac--Moody group, also denoted by $\omega$. On each subgroup $G_i\cong SL_2(F)$ it acts as transpose-inverse, on $T$ it acts by taking inverses. This defines $\omega$ on a generating set of $\mathcal G_A(F)$; one verifies that all the defining relations are indeed preserved. As a result it swaps $B^+$ and $B^-$.

\begin{definition}\label{defn:orthform}  
Let $\mathcal G_A(F)$ be a Kac--Moody group and let $\omega$ be the Chevalley involution described above. The \textbf{orthogonal form} of $G_A(F)$ is the subgroup
\[
\mathcal K_A(R)=\setcon{g\in G_A(F)}{\omega(g)=g}.
\]
\end{definition}

\section{Acylindrical hyperbolicity}\label{sect:suffacylhyp}
In this section we combine known results of \cite{BBF13} and \cite{DGO11} to give a sufficient condition for a group to be acylindrically hyperbolic which does not require any restriction on the class of groups studied. We then discuss some applications concerning groups acting on \cat spaces.

Let $(X,d)$ be a metric space and let $Y\subseteq X$. We will use $\mathcal{N}_\varepsilon(Y)$ to denote the \textbf{$\varepsilon$-neighbourhood} of $Y$, that is the set of all $x\in X$ for which there is some $y\in Y$ with $d(x,y)<\varepsilon$.

\subsection{A criterion for acylindrical hyperbolicity}
The original definition of acylindrical hyperbolicity for a group requires the existence of a non-elementary acylindrical action on a hyperbolic space \cite{Os13}. Here we will use the following equivalent definition (cf. \cite[Theorems $4.42$, $6.14$]{DGO11}). 

\begin{definition}\label{defn:acylhyp}  
Let $G$ be a group which acts by isometries on a hyperbolic metric space $Q$, and let $H$ be a subgroup of $G$. Suppose that the following conditions hold. 
\begin{enumerate}
\item $H$ acts properly on $Q$. 
\item There is some $s\in Q$ such that the $H$-orbit of $s$ is quasi-convex in $Q$.

\item $H$ is \textbf{geometrically separated}: there exists some $s\in Q$ such that for every $\varepsilon>0$ there exists an $R>0$ such that 
\[
 \tu{diam}\left(H(s) \cap \mathcal{N}_\varepsilon(gH(s))\right)\geq R
\]
implies $g\in H$.
\end{enumerate}
Then we say that $H$ is \textbf{hyperbolically embedded} in $G$. If $G$ admits an infinite, proper subgroup $H$ which  is hyperbolically embedded  (i.e. $H$ is a \textbf{non-degenerate} hyperbolically embedded subgroup in the terminology from \cite{DGO11}),  we say that $G$ is \textbf{acylindrically hyperbolic}.
\end{definition}
We note that in the above two conditions it is equivalent to say that this holds for some $s\in Q$ or for every $s\in Q$, however it is certainly not the case in general that $R$ may be chosen independently of $s$.

We are now ready to establish our criterion for acylindrical hyperbolicity.

\begin{proof}[Proof of Theorem~\ref{WPDstrcontAcylhyp}] 
By \cite[Proposition $4.7$]{BBF13}, the subgroup $H'$ consisting of all elements $g\in G$ such that $H(s)$ and $g(H(s))$ are at finite Hausdorff distance is virtually cyclic and contains $H$ as a finite index subgroup.

The hypotheses of Theorem \ref{WPDstrcontAcylhyp} are precisely those of \cite[Theorem $4.26$]{BBF13}, so we deduce that $G$ admits an action by isometries on a specific quasi-tree $Q$. To complete the proof we now show that this action satisfies the hypotheses of Definition \ref{defn:acylhyp}.

The quasi-tree $Q$ is constructed from a collection of copies of $\R$ indexed by the left cosets of $H'$ in $G$, attached in such a way that the resulting space $Q$ is connected and each $\R_{gH'}$ is \textbf{totally geodesically embedded}, i.e. the only geodesic between two points in $\R_{gH'}$ is the one inside this line. The element $h$ acts on $L=\R_{H'}$ by translation, so the action of $H$ is proper and the $H'$-orbit of a point on this line is quasi-convex.

The stabiliser of $L$ in $G$ is precisely $H'$. If $g\not\in H'$ then $gH'(s)$ is contained in $\R_{gH'}$. From the construction of $Q$ it follows that any path between two distinct copies of $\R$ must meet two balls of fixed radius with their centres on the two lines: this follows from \cite[Lemma 3.9]{BBF10} and is explicitly proved in \cite[Proposition 2.6]{Hu12}. Therefore, for any $\varepsilon>0$, the set $H'(s) \cap \mathcal{N}_\varepsilon(gH'(s))$ has finite diameter depending only on $\varepsilon$. Thus, (iii) holds. We may now apply \cite[Theorem~4.42]{DGO11} and deduce that $H'$ is hyperbolically embedded in $G$. 

Finally, the group $G$ is acylindrically hyperbolic if $H'$ is proper, so it suffices to ensure $G\neq H'$, or equivalently, that $G$ is not virtually cyclic.
\end{proof}

\subsection{Groups acting on \cat spaces with rank~$1$ elements}

An isometry $g$ of a \cat space $X$ is called a \textbf{rank $1$ isometry} if no $g$-axis bounds a half-flat. When the space $X$ is proper, every rank~$1$ isometry is strongly contracting: this is proved in \cite[Th.~5.4]{BestvinaFujiwara}. In particular, Theorem~\ref{WPDstrcontAcylhyp} readily implies  the following fact, which was first observed by A.~Sisto \cite{Sisto-contracting}. 

\begin{proposition}\label{prop:Sisto}
Let $G$ be a discrete group acting properly by isometries on a proper \cat space $X$.  If $G$ contains a rank $1$ isometry, then $G$ is acylindrically hyperbolic or virtually cyclic. \qed
\end{proposition}

The following noteworthy consequence of another result of A.~Sisto shows that if $G$ acts cocompactly, then the hyperbolically embedded cyclic subgroups essentially coincide with the cyclic subgroups generated by rank 1 elements.

\begin{proposition}\label{prop:cat0groups}
Let $G$ be a discrete group acting properly and cocompactly by isometries on a proper \cat space $X$.  Given an element $g \in G$ of infinite order, the following conditions are equivalent:
\begin{enumerate}
\item $g$ is rank~$1$.

\item $g$ is contained in a hyperbolically embedded virtually cyclic subgroup of $G$.
\end{enumerate}
\end{proposition}

\begin{proof}
As noticed above, it is a consequence of Theorem~\ref{WPDstrcontAcylhyp}  that (i) implies (ii). Conversely, if $g$ is not rank $1$, then $g$ is a hyperbolic isometry (because $g$ is of infinite order and $G$ is a discrete cocompact group of isometries) having an axis which bounds a half-flat. Since $G$ is quasi-isometric to $X$, we infer that $G$ contains quasi-geodesics joining points of $\langle g \rangle$ that do not remain within bounded distance of $\langle g \rangle$. By Theorem~1 from \cite{Sisto-Morse}, this implies that (ii) fails.
\end{proof}

The following consequence is immediate.

\begin{corollary}\label{cor:cat0groups}
Let $G$ be a discrete group acting properly and cocompactly by isometries on a proper \cat space $X$.  Assume that $G$ is not virtually cyclic (equivalently $X$ is unbounded but not quasi-isometric to the real line). Then the following conditions are equivalent:
\begin{enumerate}
\item $G$ contains a rank~$1$ isometry.

\item $G$ is acylindrically hyperbolic. \qed
\end{enumerate}
\end{corollary}

Rank~$1$ isometries of \cat spaces should be thought of as `regular'. In particular, one expects that if the full isometry group of a \cat space contains rank~$1$ elements, then any sufficiently big subgroup should also contain such. The following result provides an illustration of this paradigm. 

\begin{proposition}\label{prop:LimitSet}
Let $X$ be a proper \cat space, such that $\Isom(X)$ contains a rank~$1$ element. Let $G \leq \Isom(X)$ be a subgroup whose  limit set is the full visual boundary $\partial X$. Then $G$ also contains rank~$1$ elements.  
\end{proposition}

\begin{proof}
Let $h \in \Isom(X)$ be rank~$1$ and $\xi_+, \xi_- \in \partial X$ be its attracting and repelling fixed points. Assume first that $G$ fixes both $\xi_+$ and $\xi_-$. Then the set of geodesic lines joining $\xi_+$ to $\xi_-$ is $G$-invariant. Since $h$ is rank~$1$, the union of those lines is within a bounded neighbourhood of one of them. As $G$ is unbounded, it follows that $G$ contains a hyperbolic isometry with an axis parallel to the $h$-axes. In particular that isometry must be a rank~$1$ isometry, as desired. 

Up to replacing $h$ by its inverse, we may now assume that $G$ does not fix $\xi_+$. 
Let $(g_n)$ be a sequence in $G$ such that $g_n x$ converges to $\xi_+$ for some (hence all) $x \in X$. Upon extracting, we may assume that $g_n^{-1}x$ converges to some point $\eta \in \partial X$. If $\eta = \xi_+$, then we choose an element $\gamma \in G$ such that $\gamma \xi_+ \neq \xi_+$. Then we have $\lim_n g_n \gamma x = \xi_+$ and $\lim_n (g_n \gamma)^{-1} x =    \gamma^{-1}\xi_+ \neq \xi_+$. By Lemmas III.3.2 and III.3.3 of \cite{Ballmann}, it follows that $g_n \gamma \in G$ is rank~$1$ for all sufficiently large $n$. 
\end{proof}

\begin{corollary}\label{cor:Lattices}
Let $X$ be a proper \cat space such that $\Isom(X)$ acts cocompactly, and contains rank~$1$ elements. Then every lattice $\Gamma \leq \Isom(X)$ contains rank~$1$ elements. In particular, every lattice  $\Gamma \leq \Isom(X)$ is acylindrically hyperbolic, unless $X$ is quasi-isometric to the real line. 
\end{corollary}

\begin{proof}
By \cite[Prop.~2.9]{CaMoDiscrete}, the limit set of $\Gamma$ coincides with the limit set of $\Isom(X)$, which is full since $\Isom(X)$ acts cocompactly. Hence the desired conclusions follow from Proposition~\ref{prop:LimitSet} and Proposition~\ref{prop:Sisto}. 
\end{proof}

\begin{remark}\label{rem:KM}
Many examples of groups satisfying the hypotheses of Corollary~\ref{cor:Lattices} are provided by Kac--Moody groups over finite fields and their buildings. Indeed, if $X$ is the positive building of a Kac--Moody group $\mathcal G_A(F)$ over a finite field $F$, then $X$ is proper and $\Aut(X)$ acts cocompactly. Moreover, by \cite[Th.~1.1]{CapFuj-10-rnk1}, it contains rank~$1$ elements as soon as it is of irreducible, non-spherical and non-affine type.  If the field $F$ is large enough (i.e. of order larger than the size of the defining Cartan matrix $A$), then the subgroup $B_- \leq \mathcal G_A(F)$ is known to be a lattice in $\Aut(X)$, see \cite{RemyCRAS}. Further examples of lattices in $\Aut(X)$ can be constructed as centralisers of suitable involutory automorphisms of $\mathcal G_A(F)$, see \cite{GramlichMuhlherr}. By Corollary~\ref{cor:Lattices}, all these lattices are thus acylindrically hyperbolic. In particular they are not simple. This lies in sharp contrast with the Kac--Moody group $\mathcal G_A(F)$ itself, which is known to be mostly simple when $F$ is finite, see \cite{CapraceRemy}. 
\end{remark}

We finally record that when the space $X$ is a finite-dimensional, but not necessarily proper, \cat cube complex, then results from \cite{CapraceSageev} ensure the existence of strongly contracting elements for large families of automorphism groups of $X$. We recall that the action of a group $G \leq \Aut(X)$ on $X$ is  called \textbf{essential} if no $G$-orbit stays within a bounded distance from a half-space. 

\begin{corollary}\label{cor:CCC}
Let $X$ be a finite-dimensional (possibly non-proper) irreducible \cat cube complex and $G \leq \Aut(X)$ be a group acting essentially, without a fixed point in the visual boundary $\partial X$. If the $G$-action on $X$ is metrically proper, then $G$ is either virtually cyclic or acylindrically hyperbolic.
\end{corollary}

\begin{proof}
By \cite[Th.~6.3]{CapraceSageev}, the group $G$ contains a strongly contracting isometry. Since the action of $G$ is metrically proper, any hyperbolic element of $G$ is WPD. Therefore, the conclusion follows from Theorem~\ref{WPDstrcontAcylhyp}.
\end{proof}

\section{Regular isometries of buildings}\label{sect:regisoms}

We will now move on to groups acting on buildings that are not necessarily proper. The \cat realisation of a building is always a finite-dimensional \cat space \cite{Davis}, and the existence of strongly contracting isometries for large families of automorphism groups has been established in \cite[Th.~1.1]{CapFuj-10-rnk1}. The remaining difficulty is thus to ensure the existence of sufficiently many WPD elements, without assuming that the action of the ambient group is metrically proper. This is the content of Theorem~\ref{regularWPD}, whose proof is the focus of this chapter. 

Throughout, we let $(X, d)$ be the \cat realisation of a building of type $(W, S)$, where $S$ is finite. The nearest point projection (in the sense of  \cat geometry) to a closed convex subset $\mathcal C$ of $X$ is denoted by $\pi_{\mathcal C}$. We recall that there is also a \textbf{Weyl-distance} $\delta$ from pairs of chambers to $W$: we say $\delta(C,D)=w$ if there exists a minimal gallery $C=C_0,C_1,\dots,C_n=D$ in $X$  such that $C_{i-1}$ and $C_i$ are $s_i$-adjacent and $w=s_1\dots s_n$. The element $w \in W$ is then independent of the choice of the minimal gallery. We refer to \cite{AB08} and \cite{Davis} for general facts on the combinatorics and geometry of buildings.

\subsection{Regular points and regular lines}

A point $x \in X$ is called \textbf{regular} if it belongs to the interior of a chamber. Equivalently $x$ is regular if it is not contained in any wall of any apartment containing $x$. Given $\vareps>0$, we say that $x$ is \textbf{$\vareps$-regular} if the open ball $B(x;  \vareps)$ is entirely contained in the interior of a chamber. Equivalently $x$ is $\vareps$-regular if it lies at distance~$\geq \vareps$ from any wall of any apartment. 

Recall that a geodesic line   in $X$ is always contained in an apartment. A geodesic line $\ell$ is called \textbf{regular} if there exists an apartment $\mathcal A$ containing $\ell$ such that $\ell$ is not contained in a bounded neighbourhood of any wall of $\mathcal A$. It follows that $\ell$ contains regular points of $\mathcal A$. In particular $\ell$ meets the interior of some chamber $c$ of $\mathcal A$. Now, given any apartment $\mathcal A'$ containing $\ell$, the retraction $\rho_{\mathcal A, c}$ induces an isometry $\mathcal A' \to \mathcal A$ fixing $\ell$. This implies that $\ell$ is not contained in the bounded neighbourhood of any wall of $\mathcal A'$. Thus the defining property of $\ell$ holds for all apartments containing it.

Two geodesic lines are called \textbf{parallel} if they lie at bounded Hausdorff distance of each other or, equivalently, if they have the same endpoints in the visual boundary. The following fact is an important property of regular lines. Before stating it, we recall that a set of chambers in a building is called \textbf{combinatorially convex} if each minimal gallery joining two chambers of that set is entirely contained in the set. For example, apartments are combinatorially convex.

\begin{lemma}\label{lem:bi-infinite}
Let $\ell$ be a regular geodesic line contained in an apartment $\mathcal A$. For any geodesic line $\ell'$ parallel to $\ell$, the combinatorial convex hulls of $\ell$ and $\ell'$ coincide. In particular $\ell'$ is also contained in $\mathcal A$, and $\ell'$ is regular. 
\end{lemma}

\begin{proof} 
Let $\mathcal C$ be the combinatorial convex hull of $\ell$. We have $\mathcal C \subseteq \mathcal A$. Moreover $\mathcal C$ is closed and convex in the CAT($0$)-sense, since $\mathcal C$ is the intersection of all half-apartments containing $\ell$.

By the Flat Strip Theorem \cite[$2.13$]{BH99} the convex hull of $\ell \cup \ell'$ is a Euclidean rectangle $R=\ell \times I$ for some geodesic segment $I$ joining a point $x$ on $\ell $ to its closest point projection $x'$ on $\ell'$. The set $R\cap \mathcal C$ is closed and convex and contains $\ell $. Therefore, it is of the form $\ell \times J$ where $J = [x, y]$ is some closed subinterval of $I$ containing $x$.

We next claim that $I=J$. Indeed, suppose the contrary. Let then $\ell''$ be the geodesic line parallel to $\ell$ and containing $y$. We now show that $\ell''$ is not contained in the union of all walls of $\mathcal A$. If there is some wall $M$ in $\mathcal A$ containing at least two points of $\ell''$ then $\ell''$ is contained in $M$ by \cite[Lem.~3.4]{NV02}, so as $\ell$ is regular, only countably many points of $\ell''$ lie on walls in $A$.

Up to replacing $x$ by a neighbouring point on $\ell$, we may therefore assume that $y$ does not lie on any wall in $\mathcal A$, and is thus a regular point.  Therefore there is some small ball around $y$ contained in a chamber, hence in $\mathcal C$, contradicting the fact that $y$ lies in the boundary of $R\cap \mathcal C$. 

The claim implies that $\ell' \subset \mathcal C$, so that $\mathcal C' \subseteq \mathcal C$ and $\ell'$ is regular. By symmetry we also have $\mathcal C \subseteq \mathcal C'$, and we are done.
\end{proof}

\subsection{Regular automorphisms and combinatorial hulls}

An automorphism $h \in \Aut(X)$ is called \textbf{regular} if it is hyperbolic and has an axis which is a regular line. In that case all its axes are regular by Lemma~\ref{lem:bi-infinite}. We define the \textbf{combinatorial hull} $\mathcal{H}$ of $h$ as the combinatorial convex hull of some fixed axis $\ell$ of $h$. By Lemma~\ref{lem:bi-infinite}, the combinatorial hull does not depend on the choice of the axis $\ell$, and it contains all lines parallel to $\ell$, in particular all $h$-axes. 

The main result of this section is that combinatorial hulls of regular elements enjoy a strong `attracting property'.

\begin{proposition}\label{prop:regularSegment}
Let $h \in \Aut(X)$ be a regular element and $\mathcal H$ be its combinatorial hull.  Given $C \geq 0$, there exists $T >0$ such that the following holds for all $x, y \in X$:

If $d(x, \ell), d(y, \ell) < C$ and $d(x, y)>2T$, then $[x, y] \cap \mathcal H$ contains a geodesic segment $[x',y']$ such that $d(x,x'),d(y,y')<T$. 
\end{proposition}

The remainder of the section deals with the proof of Proposition \ref{prop:regularSegment}, which we build up through a collection of lemmas studying the properties of geodesics neighbouring a regular axis. Throughout this it will be important to keep track of regular points. 
The first relevant result was Lemma~\ref{lem:bi-infinite} above on geodesic lines. We will now move on to   geodesic rays.

\medskip
The following notation is fixed for the rest of this section. We let $h \in \Aut(X)$ be a regular element with axis $\ell$,  and $\mathcal H$ be its combinatorial hull. Then $\mathcal H$ is combinatorially convex, hence it is an intersection of half-apartments, and is thus also closed and convex for the \cat metric. This will be used frequently below, without further comment. We denote by $\xi_+$ (resp. $\xi_-$) the attracting (resp. repelling) fixed point of $h$ in the visual boundary $\partial X$. 

\begin{lemma}\label{lem:semi-infinite}
Let $\rho \colon \R_+ \to X$ be a geodesic ray pointing to $\xi_-$. Then there exists a constant $T\geq 0$ such that $\rho(t) \in \mathcal H$ for all $t \geq T$. 
\end{lemma}

\begin{proof}
Assume the contrary. Then $\rho(t) \not \in \mathcal H$ for any $t \in \R_+$. Set 
$$D = \inf_{t \geq 0} d(\rho(t), \mathcal H).$$
Since $\mathcal H$  is \cat-convex, the distance to $\mathcal H$ is a convex function on $X$, see \cite[Cor.~II.2.5]{BH99}. In particular the map $\R_+ \to \R_+ : t \mapsto d(\rho(t), \mathcal H)$ is non-increasing. 

Upon replacing $h$ by a suitable positive power of itself, we may assume that the translation length of $h$ is larger than $2$. This implies that the sequence $(h^n \rho(n))_{n \geq 0}$ converges to $\xi_+$. Denote by $\pi_{\mathcal H}$ the nearest point projection to $\mathcal H$ and set 
$$z_n = \pi_{\mathcal H} h^n \rho(n).$$
Since $\mathcal H$ is a closed subset of an apartment, it is a proper space. It follows that, after extracting a subsequence,  the sequence of geodesic rays $[z_n, \rho_-)$ converges uniformly on compact sets to a geodesic line $\ell' \subset \mathcal H$ which is parallel to $\ell$. In particular $\ell'$ is regular. Pick a regular point $q \in \ell'$. Then there exists $\vareps >0$ such that the ball $B(q, 3\vareps)$ is entirely contained in $\mathcal H$. 

By construction, there is $N \geq 0$ such that for all $n \geq N$, there is a point $q_n \in [z_n, \xi_-)$ with $d(q_n, q) < \vareps$. Moreover, since the map $t \mapsto d(\rho(t), \mathcal H)$ is non-increasing, we find some $M \geq 0$ such that for all $n \geq M$, $d( \rho(n),    \pi_{\mathcal H} \rho(n))
 < D + \vareps$. The automorphism $h$ commutes with the projection $\pi_{\mathcal H}$, since $h$ stabilises $\mathcal H$, so we deduce that
\begin{align*}
d(h^n \rho(n), z_n) 
& = d(h^n \rho(n), \pi_{\mathcal H} h^n \rho(n))\\
&= d(h^n \rho(n), h^n  \pi_{\mathcal H} \rho(n))\\
&= d(  \rho(n),    \pi_{\mathcal H} \rho(n))\\
& < D + \vareps
\end{align*}
for all $n \geq M$. In particular, for $n \geq \max \{M, N\}$, there is a point $q'_n \in [h^n \rho(n), \xi_-)$ with $d(q'_n, q_n)< D + \vareps$. We have $d(q'_n, q) \leq d(q'_n, q_n) + d(q_n, q) < D + 2\vareps$. Since $B(q, 3\vareps) \subset \mathcal H$, we infer that the geodesic segment $[q'_n, q]$ contains points of $\mathcal H$ at distance $< D-\vareps$ from $q'_n$. On the other hand, we have
$$\inf_{t \geq 0} d(h^n \rho(n +t), \mathcal H) = \inf_{t \geq 0} d(\rho(n+t), \mathcal H) = D,$$
so that $d(q'_n, \mathcal H) \geq D$, a contradiction.
\end{proof}

Since the geodesic line $\ell$ is periodic and regular, it is easy to see that any geodesic ray in an apartment $\mathcal A$ containing $\ell$ that remains in a bounded neighbourhood of $\ell$, is itself regular. The following subsidiary fact provides a quantitative version of that fact, ensuring in particular that for a fixed small $\vareps >0$, the collection of $\vareps$-regular points on that geodesic ray is equidistributed.

\begin{lemma}\label{lem:Apt}
Let $\mathcal A$ be an apartment containing $\ell$. For any  $C> 0$, there exist $\vareps \in (0, 1)$ and $L \geq 1$ such that the following holds for all $x, y \in \mathcal A$ :

If $d(x, \ell), d(y, \ell) < C$ and $d(x, y) \geq L$, then there is $z \in [x, y]$ which is $\vareps$-regular. 
\end{lemma}

\begin{proof}
Assume the contrary. Then there exist sequences $(x_n)$ and $(y_n)$ in $\mathcal A$ such that $d(x_n, \ell), d(y_n, \ell) < C$, $d(x_n, y_n) \geq n$ and every point $z \in [x_n, y_n]$ is $1/n$-close to a wall of $\mathcal A$. Using the fact that the cyclic group $\langle h \rangle $ acts cocompactly on $\ell$, we may assume that the midpoint of $[x_n, y_n]$ remains in a bounded neighbourhood of a base point $p \in \ell$. Since $\mathcal A$ is proper, upon extracting a subsequence we may assume that $[x_n, y_n]$ converges to a geodesic line $\ell'$ which is parallel to $\ell$. By construction every point of $[x_n, y_n]$ is $1/n$-close to  some wall of $\mathcal A$. It follows that $\ell'$ does not contain any regular point. In other words  every point of $\ell'$ lies on a wall of $\mathcal A$. Since the collection of walls of $\mathcal A$ is locally finite, we see that two points of $\ell'$ must lie on the same wall. By \cite[Lem.~3.4]{NV02}, this implies that $\ell'$ is entirely contained in a wall, which contradicts the assumption that $\ell$ is regular. 
\end{proof}

Our next goal is to show that the constant $T$ in Lemma~\ref{lem:semi-infinite} is bounded above by a constant which depends only on the distance from $\rho(0)$ to $\ell$. This is achieved by the following. 

\begin{lemma}\label{lem:semi-infinite2}
Let $p \in \ell$. For each $C \geq 0$, there is some $T=T(C) \geq 0$ such that the following holds:  

For any geodesic ray $\rho \colon \R_+ \to X$ pointing to $\xi_-$ with $d(\rho(0), p) < C$, we have $\rho(t) \in \mathcal H$ for all $t \geq T$. 
\end{lemma}

\begin{proof}
Assume the contrary. Then there is $C \geq 0$ and, for each $n >0$, a geodesic ray $\rho_n \colon \R_+ \to X$ pointing to $\xi_-$ with $d(\rho_n(0), p)< C$ such that $\rho_n(t) \not \in \mathcal H$ for all $t \in [0, n]$. Let $x_n = \rho_n(0)$. By Lemma~\ref{lem:semi-infinite}, the intersection $[x_n, \xi_-) \cap \mathcal H$ is a ray, say $[y_n, \xi_-)$. Hence $d(x_n, y_n) > n$. 

Set $x'_n = \pi_{\mathcal H}(x_n)$. Recall that $\mathcal H$ is a subset of an apartment. Applying Lemma~\ref{lem:Apt}, we find $\vareps\in (0, 1)$ 	and $L \geq 1$ such that for all $n > 2L$,  the geodesic segment $[x'_n, y_n]$ contains an $\vareps$-regular point $q_n$ such that $L \leq d(y_n, q_n) \leq 2L$. In particular $B(q_n, \vareps) \subset \mathcal H$. 

We next observe that  $d(x_n, x'_n) \leq d(x_n, p) < C$. Therefore, by the \cat inequality for the geodesic triangle $(x_n, x'_n, y_n)$, we see that for $n$ sufficiently large, the point $q_n$ is $\vareps/2$-close to a point $q'_n$ belonging to the segment $ [x_n, y_n]$.  Since $q_n$ is $\vareps$-regular, it follows that $q'_n \in \mathcal H$, whence $q'_n = y_n$ since by construction we have $[x_n, y_n] \cap \mathcal H = \{y_n\}$. This implies that $d(y_n, q_n) \leq \vareps/2$, so that $L \leq \vareps/2<1$, a contradiction. 
\end{proof}

We are now ready to give the proof of Proposition~\ref{prop:regularSegment}.

\begin{proof}[Proof of Proposition~\ref{prop:regularSegment}]
Assume the contrary. Then there is a sequence $M_n> n$, and two sequences $(x_n), (y_n)$ in $X$ such that $d(x_n, \ell), d(y_n, \ell) < C$ and $d(x_n, y_n)= M_n$ for all $n$ such that $[x_n, y_n] \cap \mathcal H$ is a geodesic segment of length at most~$M_n-n$.

The cyclic group $\langle h \rangle $ acts cocompactly on $\ell$. Therefore, there is no loss of generality in assuming that $(x_n)$ remains in a bounded neighbourhood of some base point $p \in \ell$. Moreover, upon replacing $h$ by its inverse, we may then assume that $(y_n)$ converges to $\xi_-$. 

Let $\rho_n$ be the geodesic ray joining $x_n$ to $\xi_-$.  By Lemma~\ref{lem:semi-infinite2}, there exists $T$ such that $\rho_n(t) \in \mathcal H$ for all $t \geq T$. By Lemma~\ref{lem:Apt}, there exists $\vareps \in (0, 1)$, $L\geq 1$ and $t_n \in [T, T+L]$ such that 
$\rho_n(t_n)$ is $\varepsilon$-regular.  Set $q_n = \rho_n(t_n)$, so $B(q_n; \vareps) \subset \mathcal H$. 

Now let $y'_n$ be the projection of $y_n$ to the ray $\{\rho_n(t)\}_{t \geq T}$. Since $d(y_n, \ell) < C$, it follows that $d(y_n, y'_n)$ is uniformly bounded. From the \cat inequality in the triangle $(x_n, y_n, y'_n)$, we infer that for $n$ sufficiently large, the geodesic segment $[x_n, y_n]$ contains a point $q'_n$ which is $\vareps/2$-close  to $q_n$. Therefore $q'_n \in \mathcal H$. Moreover, $d(x_n, q'_n) \leq d(x_n, q_n) + \vareps/2 \leq T + L + \vareps/2$. Notice that the right-hand-side of the preceding inequality is independent of $n$. 

By interchanging the roles of $x_n$ and $y_n$ in the above argument, we find a point $r'_n \in [x_n, y_n] \cap \mathcal H$ such that $d(y_n, r'_n) \leq T + L + \vareps/2$ for all $n$ sufficiently large. It follows that the length of the geodesic segment $[x_n, y_n] \cap \mathcal H\supseteq [q'_n,r'_n]$ is at least $M_n - 2T - 2L - \vareps$ for  all $n$ sufficiently large. This is a contradiction.
\end{proof}

We finish this subsection by recording a subsidiary fact on combinatorial hulls. 

\begin{lemma}\label{lem:ConvexHull}
Let $b \subset \mathcal H$ be a chamber. For each $C \geq 0$, there exists $M = M(C) \geq 0$ such that for any two chambers $a, a' \subset \mathcal H$ with $d(a, \ell) < C$, $d(a', \ell) < C$ and $d(a, a')\geq  M$, the combinatorial convex hull of $\{a, a'\}$ contains a chamber in the $\la h \ra$-orbit of $b$.
\end{lemma}
\begin{proof}
Assume the contrary. Then there exist sequences $(a_n)$ and $(a'_n)$ of chambers in $\mathcal H$ such that $d(a_n, \ell) < C$, $d(a'_n, \ell) < C$ and $d(a_n, a'_n) > n$, such that the convex hull of $\{a_n, a'_n\}$ does not contain any chamber  in the $\la h \ra$-orbit of $b$. Since $\mathcal H$ is locally finite, we may assume, upon extracting and replacing $a_n$ and $a'_n$ by other chambers in the same $\la h \ra$-orbit, that $a_n = a$ for all $n$. Upon extracting further, we may assume that $a'_n \to \xi_+$. Let then $\mathcal C_n$ denote the combinatorial convex hull of $\{a, a'_n\}$. Using again the local finiteness of $\mathcal H$, we may extract once more in order to ensure that for all $r >0$, there exists $N(r)$ such that for all $m, n \geq N(r)$, we have $B(a, r) \cap \mathcal C_m = B(a, r) \cap \mathcal C_n$, where $B(a, r)$ denotes the ball of combinatorial radius $r$ around $a$ in $\mathcal H$. 

We next set $\mathcal C = \bigcup_{k \geq 0} \bigcap_{n \geq k} \mathcal C_n$. Then $\mathcal C$ is combinatorially convex, since it is the union of an ascending chain of combinatorially convex subsets. In particular $\mathcal C$ is also closed and convex in the \cat-sense. Observe moreover that the visual boundary $\partial \mathcal C$ contains the point $\xi_+$. Indeed, let $\rho$ be  a geodesic ray emanating from a point in the chamber $a$ and pointing to $\xi_+$. Since  $a'_n \to \xi_+$ it follows that the intersection $\rho \cap B(a, r) \cap \mathcal C_{N(r)}$  is a geodesic segment whose length tends to infinity with $r$. Since $B(a, r) \cap \mathcal C_{N(r)} \subset \mathcal C$ by construction, it follows that $\mathcal C$ contains $\rho$, thereby confirming that $\xi_+ \in \partial \mathcal C$. 

Now for each $n \geq 0$, we set $\mathcal D_n =  h^{-n} \mathcal C$. We have $\mathcal D_n \subset \mathcal H$ and $\xi_+ \in \partial \mathcal D_n$ for all $n$. Using the local finiteness of $\mathcal H$ once more, we may extract a subsequence of $(\mathcal D_n)$  to ensure that for all $r >0$, there exists $N'(r)$ such that for all $m, n \geq N'(r)$, we have $B(a, r) \cap \mathcal D_m = B(a, r) \cap \mathcal D_n$. We then set $\mathcal D =  \bigcup_{k \geq 0} \bigcap_{n \geq k} \mathcal D_n$. Arguing similarly as above, we see that $\mathcal D$ is combinatorially convex, hence \cat-convex, and that $\partial \mathcal D$ contains $\xi_-$. Hence $\{\xi_+, \xi_-\} \subset \partial \mathcal D$, so that $\mathcal D$ contains a geodesic line parallel to $\ell$.  From Lemma~\ref{lem:bi-infinite} it follows that $\mathcal D = \mathcal H$. Hence there is some $m$ such that $b \subset \mathcal D_m$. Therefore the chamber $h^m b$ is contained in $\mathcal C$, hence in $\mathcal C_n$ for some $n$. This is a contradiction.
\end{proof}

\subsection{When regular automorphisms are WPD}

We now complete the proof of our criterion for an automorphism of a building to be WPD. 

\begin{proof}[Proof of Theorem~\ref{regularWPD}]
Let $h \in G$ be a regular automorphism with axis $\ell$, and $\mathcal H$ denote its combinatorial hull. 
By hypothesis there exists a chamber $b$ in $\mathcal H$ whose  stabiliser $G_b$ is finite.

\medskip
We now fix some $D>0$ and some $x\in X$. Applying Proposition \ref{prop:regularSegment} with $C =  d(x,\ell)+D$ we obtain a constant $T=T_{x,D}$ such that, for every $m$ sufficiently large and every pair of points $y,z$ with $d(x,y),d(h^mx,z)<D$ the geodesic segment $[y, z] \cap \mathcal H$ contains a geodesic segment $[y',z']$ with $d(y,y'),d(z,z')<T$. Applying Lemma~\ref{lem:Apt} with the same constant $C$, we get constants $\vareps \in (0, 1)$ and $L \geq 1$ such that every geodesic segment of length~$\geq L$ contained in the $C$-neighbourhood of $\ell$ in $\mathcal H$ contains an $\vareps$-regular point. Upon enlarging $m$, we may assume that $d(x, h^m x) > M+2 \delta + 2L+2T$, where $M = M(C)$ is the constant afforded by applying Lemma~\ref{lem:ConvexHull} to the chamber $b \subset \mathcal H$, and $\delta$ denotes the diameter of a chamber. 

Let $P = \setcon{g\in G}{ d(x,gx)<D \text{  and }d(h^mx,gh^mx)<D}$. We need to show that $P$ is finite. 

To this end, let $x' \in [x, h^m x]$ (resp.  $y' \in [x, h^m x]$) be the unique point with $d(x, x')=T$ (resp. $d(y', h^mx) = T$). By Proposition~\ref{prop:regularSegment}, we have $[x', y'] \subset \mathcal H$. Since $d(x', y')> M+2\delta +2L$, there exist two $\vareps$-regular points $p,p' \in [x', y']$ such that $d(p, x') \leq L$ and $d(p', y') \leq L$.  In particular we have $d(p, x) \leq T+L$ and $d(p', h^m x)\leq T+L$. Moreover $d(p, p') \geq M+2\delta$ since $d(x, h^m x) > M  +2\delta +2L+2T$.  

Let $a, a' \subset \mathcal H$ be the unique chambers containing $p, p'$ respectively, so that $d(a, a')\geq M$. Hence Lemma~\ref{lem:ConvexHull} ensures that the convex hull of $\{a, a'\}$ contains a chamber $b'$ in the $\la h \ra$-orbit of $b$. In particular $b'$ has finite stabiliser. 

We claim that $P \subseteq \setcon{g \in G}{ga \cup ga' \subset \mathcal H \text{ and } \max\{d(a, ga), d(a', ga') \}\leq D}$. 

Indeed, let $g \in P$. Then $d(gx, \ell) <C$ and $d(gh^m x, \ell)<C$. Thus Proposition~\ref{prop:regularSegment} yields a segment $[x'_g, y'_g] \subset [gx, gh^mx] \cap \mathcal H$ such that $d(gx, x'_g)< T$ and $d(gh^m x, y'_g)< T$. By the definition of $x'$ and $y'$, it follows that  $[gx', gy'] \subset [x'_g, y'_g] \subset \mathcal H$, since $d(x, x')= d(y', h^m x) =T$. In particular we have $gp, gp' \in [x'_g, y'_g]$, so that  $gp, gp' \in \mathcal H$.  The points $gp$ and $gp'$ are   regular, so we infer that $ga$ and $ga'$ are both chambers contained in $\mathcal H$. 

Recalling that the displacement function $z \mapsto d(z, gz)$ is convex (see \cite[Prop.~II.6.2(3)]{BH99}), we have $\max\{d(p, gp), d(p', gp'\} \leq \max \{d(x, gx), d(h^m x, gh^m x)\} \leq D$. Hence we deduce that $\max\{d(a, ga), d(a', ga')\} \leq D$ and the claim is verified. 

The convex hull of $\{a, a'\}$ is bounded and contains the chamber $b'$, so the claim implies that there exists a constant $D'$ such that $P \subseteq \setcon{g \in G}{gb' \subset \mathcal H \text{ and } d(b', gb') \leq D'}$.

Since $\mathcal H$ is proper and the stabiliser of $b'$ is finite, we conclude that $P$ is indeed finite.  Hence $h$ is WPD.
\end{proof}

We record the following immediate consequence of Theorems~\ref{WPDstrcontAcylhyp} and~\ref{regularWPD}. 

\begin{corollary}\label{cor:AcylBuilding1}
Let $X$ be a building of type $(W, S)$, with $S$ finite, and let $G \leq \Aut(X)$ be a group acting with finite chamber stabilisers. If $G$ contains a strongly contracting, regular automorphism, then $G$ is either virtually cyclic or acylindrically hyperbolic. \qed
\end{corollary}

A sufficient condition for $G$ to act with finite chamber stabilisers is that the $G$-action is metrically proper. However, the latter condition is much stronger: for example, the orthogonal forms of real Kac--Moody groups will be shown below to act on their building with finite chamber stabilisers, but their action is not metrically proper. Similarly, the action of a graph product of infinite groups on the associated right-angled building (see Proposition~\ref{prop:GraphProd} below) is not metrically proper.

\section{Acylindrical hyperbolicity for groups acting on buildings}\label{sec:AcylBuilding}

The goal of this chapter is to prove Theorem~\ref{thm:AcylBuilding}. The subtlety here is in ensuring that a strongly contracting axis is regular. We begin by studying this in the simplest case of the action of a Coxeter group on its Davis complex.

\subsection{Regular strongly contracting elements of Coxeter groups}

\begin{proposition}\label{prop:Coxeter}
Let $(W, S)$ be an irreducible, non-spherical, non-affine Coxeter system with $S$ finite. Then $W$ contains an element $w$ acting as a regular, strongly contracting, automorphism of the Davis complex $X$. 
\end{proposition}

\begin{proof}
It is proved in \cite[Cor.~4.7]{CapFuj-10-rnk1} that $W$ contains rank~$1$ elements: in fact, any \textbf{Coxeter element}, defined as the product of all elements of $S$ taken in an arbitrary order, is rank~$1$. Since $X$ is proper, rank~$1$ elements are strongly contracting (see \cite[Th.~5.4]{BestvinaFujiwara}). However, an extra argument is needed to ensure the existence of such elements that are also \emph{regular}. One possible way to do so would be to  show that the Coxeter elements are regular, i.e. no non-zero power of a Coxeter element stabilises a wall. This is proved to be the case in \cite[Lem.~3.4]{Marquis} under the extra hypothesis that $W$ is the Weyl group of a Kac--Moody algebra, i.e. $W$ is \textbf{crystallographic}, in the sense that the only Coxeter numbers involved in the Coxeter presentation of  $W$ with respect to $S$ belong to the set $\{2, 3, 4, 6, \infty\}$. It is very likely that Coxeter elements are regular in all  irreducible, non-spherical, non-affine Coxeter groups. Nevertheless, in order to deal with general Coxeter groups, we will follow an alternative approach, using \cat cube complexes. Notice however that the case of crystallographic Coxeter groups is the only relevant one for the Main Theorem and the Corollary from the introduction. 

By Lemma~\ref{lem:NR} below, the $W$-action on its Niblo--Reeves \cat cube complex $Y$ is essential, without a fixed point at infinity. Moreover $Y$ is irreducible and  finite-dimensional. It then follows from \cite[Prop.~5.1]{CapraceSageev} that $Y$ contains a pair of distinct hyperplanes $\hat h$, $\hat h'$ such that no hyperplane crosses both of them. The same property therefore holds in the Davis complex $X$. Now let $g \in W$ be the product of the two reflections associated with $\hat h$ and $\hat h'$. Let $V^+$ be a neighbourhood of its attracting fixed point in $\partial X$ (resp. $V^-$ be a neighbourhood of its repelling fixed point) containing no boundary point of $\hat h$ or $\hat h'$. By \cite[Prop.~3.5 and Cor.~4.7]{CapFuj-10-rnk1} there exists a rank~$1$ element $w \in W$ whose attracting and repelling fixed points belong respectively to $V^+$ and $V^-$. If $w$ were not regular, then its attracting and repelling fixed points would belong to the boundary of a hyperplane. By construction, that hyperplane would cross both $\hat h$ and $\hat h'$, which is a contradiction.
\end{proof}

\begin{lemma}\label{lem:NR}
Let $(W, S)$ be an irreducible, non-spherical, non-affine Coxeter system with $S$ finite and let $Y$ be the associated \cat cube complex, as constructed by Niblo--Reeves \cite{NibloReeves}. Then $Y$ is an irreducible, proper, finite-dimensional cube complex. Moreover the $W$-action on $Y$ is proper, essential, without a fixed point in $\partial Y$. 
\end{lemma}

\begin{proof}
The assertions that $Y$ is proper, finite-dimensional, and that the $W$-action is proper, are proved in \cite{NibloReeves}. That $W$ acts essentially follows from \cite[Lem.~2.19]{CapraceMarquis}, which ensures that every hyperplane belongs to an infinite set consisting of those hyperplanes bounded by an infinite chain of half-spaces. Given a point $\xi \in \partial Y$, consider a hyperplane $\hat h$ which is transverse to a geodesic ray pointing to $\xi$. Then the reflection of $W$ associated with $\hat h$ swaps the corresponding two half-spaces of $Y$, and therefore it does not fix $\xi$. 

It remains to show that $Y$ is irreducible. If that were not the case, then by \cite[Lem.~2.5]{CapraceSageev} the set of hyperplanes of $Y$ would be partitioned into two subsets $\mathcal H_1 \sqcup \mathcal H_2$ such that every hyperplane in $\mathcal H_1$ crosses every hyperplane in $\mathcal H_2$. By \cite[Cor.~F]{CapraceMarquis}, there is a pair of reflections $r, r'$ in $W$ such that the group $\la r, r'\ra$ is not contained in any proper parabolic subgroup of $W$. The two hyperplanes $\hat h$ and $\hat h'$ respectively stabilised by $r$ and $r'$ must be disjoint. Therefore, upon exchanging $\mathcal H_1$ and $\mathcal H_2$, we may assume that they both belong to $\mathcal H_1$. In particular so do all hyperplanes in the $\la r, r'\ra$-orbit of $\hat h$. As remarked above, every hyperplane belongs to an infinite set consisting of those hyperplanes bounded by an infinite chain of half-spaces. In particular $\mathcal H_2$ must also contain such an infinite set of hyperplanes. We now invoke the Grid Lemma from \cite[Lem.~2.8]{CapraceMarquis}. This implies that $W$ is of affine type, which is a contradiction.  
\end{proof}

\subsection{Straight elements and their combinatorial hulls}

Let $(W, S)$ be a Coxeter system. An element $w \in W$ is called \textbf{straight} if $\ell_S(w^n) = n \ell_S(w)$ for all $n$, where $\ell_S$ denotes the word length with respect to the generating set $S$. Various criteria ensuring that an element is straight are provided by T.~Marquis' paper \cite{Marquis-straight}. We extract the following. 

\begin{lemma}\label{lem:Marquis}
Let $w \in W$ be an element acting  as a regular automorphism of the Davis complex of $(W,S)$. Then $w$ is straight if and only if it is of minimal length in its conjugacy class. 
\end{lemma}
\begin{proof}
By Lemma~4.1 from \cite{Marquis-straight}, any straight element is of minimal length in its conjugacy class. 

Suppose conversely that $w$ is of minimal length in its conjugacy class. 
Let $\Sigma$ denote the Davis complex, and $c_0$ be the fundamental chamber (corresponding to the trivial element of $W$).  Since $w$ is regular, it has an axis meeting the interior of  some chamber $c$ of $\Sigma$. Therefore $w$ has a conjugate element $v$ with an axis meeting the interior of $c_0$. By Remark~4.4 from \cite{Marquis-straight}, it follows that $v$ is straight. Thus $v$ is  also of minimal length within its conjugacy class. Hence $\ell_S(v)=\ell_S(w)$. Therefore, Lemma~4.2 from \cite{Marquis-straight} ensures that $w$ is straight as well. 
\end{proof}

We also record a criterion ensuring that a chamber of a building is contained in the combinatorial hull of a given automorphism. 

\begin{lemma}\label{lem:StraightWeylDist}
Let $X$ be a building of type $(W,S)$. Let $h \in \Aut(X)$ and let $b$ be a chamber of $X$ such that the Weyl-distance $w = \delta(b, hb) \in W$ is straight, of minimal length in its conjugacy class, and is a regular isometry of the Davis complex of $(W, S)$. Then the following hold. 
\begin{enumerate}
\item $h$ is a regular automorphism of $X$.

\item  $b$ is contained in the combinatorial hull of $h$. 

\item If $w$ is strongly contracting on the Davis complex of $(W, S)$, then $h$ is strongly contracting on $X$.
\end{enumerate}
\end{lemma}
\begin{proof}
We first claim for that all $m < n \in \Z$, we have $\delta(h^m b, h^nb) = w^{n-m}$. Indeed, there exists a gallery of length $(n-m) \ell_S(w)$ joining $h^m b$ to $h^nb$, whose type is the word $w^{n-m}$. Since $w$ is straight, it follows that the word $w^{n-m}$ is reduced. Therefore the gallery we constructed is minimal. The claim follows.

Therefore the set $\{h^n b \; | \; n \in \Z\}$ is isometric (with respect to the Weyl-distance) to a subset of an apartment, and is hence contained in some apartment, say $\mathcal A$. Let $\rho = \rho_{\mathcal A, b}$ be the retraction onto $\mathcal A$ based at $b$ and let $\gamma = \rho \circ h$. We see that $\gamma$ belongs to the automorphism group of the apartment $\mathcal A$, which is isomorphic to $W$. Viewing $b$ as the fundamental chamber of $\mathcal A$, we obtain a specific isomorphism $\Aut(\mathcal A) \to W$ which maps each chamber $b'$ to the unique element of $W$ mapping $b$ to $b'$. In particular this isomorphism maps $\gamma$ to $w$. We view that isomorphism as an identification, so $\gamma = w$. 

Notice that $\bigcap_{n \in \Z} h^n \mathcal A$ is non-empty (because it contains $\setcon{h^n b}{n \in \Z}$), closed, CAT($0$)-convex and $\fgen{h}$-invariant. Therefore it contains an $h$-axis, say $\ell$. Since the retraction $\rho$ fixes $\mathcal A$ pointwise, it follows that $\ell$ is also a $w$-axis. By assumption, $w$ is regular, so we see  that $\ell$ is a regular line, hence $h$ is regular. This proves (i).

\medskip
Since $\bigcap_{n \in \Z} h^n \mathcal A$ is also combinatorially convex, and contains $\ell$, it contains the combinatorial hull $\mathcal H$ of $h$ in $X$. In particular $\mathcal H \subseteq \mathcal A$. Moreover, $h$ and $w$ share a common axis, implying that $\mathcal H$ is also the combinatorial hull of $w$. Therefore, it suffices to show that $b$ belongs to the combinatorial hull $\mathcal H_w$ of $w$. 

Assume that this is not the case. Then there exists a half-apartment $\mathfrak h$ containing $b$ but no chamber of $\mathcal H_w$. Since the maximal number of pairwise intersecting walls in $\mathcal A$ is bounded (this amounts to saying that the Niblo--Reeves cube complex $Y$ is finite-dimensional, see Lemma~\ref{lem:NR}), there is some $m>0$ such that the half-apartments $(w^{mi}\mathfrak h)_{i \in \Z}$ are pairwise disjoint. In particular the wall $\partial (w^m\mathfrak h)$ does not separate $b$ from $w^{2m}b$, so the chamber $w^mb$ does not lie on any minimal gallery between $b$ and $w^{2m} b$. This implies that $\ell_S(w^{2m}) < 2 \ell_S(w^m)$, contradicting that $w$ is straight. This proves (ii).

\medskip
For Assertion~(iii), observe that if $h$ stabilises a residue $R$ of type $J$, where $J$ is a subset of $S$, then $\ell \subset R$ hence $\mathcal H \subset \mathcal R$, since residues are combinatorially convex. Thus $R$ is the unique residue of type $J$ containing $\mathcal H$. Hence $R \cap \mathcal A$ is the unique residue of type $J$ of the apartment $\mathcal A$ containing $\mathcal H$. Hence $R \cap \mathcal A$ is stabilised by $w$. The desired assertion now follows from Theorem~5.1 in \cite{CapFuj-10-rnk1}.
\end{proof}

\subsection{Strongly contracting WPD automorphisms of buildings}

\begin{proof}[Proof of Theorem~\ref{thm:AcylBuilding}]
Let $b$ be a chamber with finite stabiliser. By hypothesis the $G$-orbit of $b$ contains a subset of an apartment which is the $W_0$-orbit of a chamber, where $W_0$ is some finite index subgroup of the Weyl group $W$. We may assume that $W_0$ is normal in $W$.

Let $w \in W$ be an element acting as a regular, strongly contracting automorphism of the Davis complex of $(W, S)$, see Proposition~\ref{prop:Coxeter}. Upon replacing $w$ by a suitable power, we may assume that $w \in W_0$. Replacing $w$ by a conjugate, we may also assume that $w$ is of minimal length in its conjugacy class. By Lemma~\ref{lem:Marquis}, it follows that $w$ is straight. 

By hypothesis, the image of the map $g\mapsto\delta(b, gb)$ contains $W_0$. In particular, there exists some $h \in G$ with $\delta(b, hb) = w$. By Lemma~\ref{lem:StraightWeylDist} the element $h$ is regular and strongly contracting, and $b$ is contained in its combinatorial hull. Theorem~\ref{regularWPD} implies that $h$ is WPD. Now Theorem~\ref{WPDstrcontAcylhyp} ensures that $G$ is acylindrically hyperbolic or virtually cyclic. But the latter case is impossible, because it would imply that $G$, and hence also $W$, is $0$- or $2$-ended.
\end{proof}

\section{Applications}

\subsection{Orthogonal forms of Kac--Moody groups}\label{sect:othformaction}
The goal of this section is to prove Theorem~\ref{findeltOrthform}.

\begin{proposition}\label{findeltOrthform} 
Let $A$ be a generalized Cartan matrix of irreducible, non-spherical, non-affine type, and $\mathcal G_A(F)$ be a Kac--Moody group of  type $A$ over a field $F$.  Let $X^+$ be one of the two factors of the twin building associated with $\mathcal  G_A(F)$ and $\mathcal A^+$ be the standard apartment of $X^+$. 

Then $\Stab_{\mathcal K_A(F)}(\mathcal A^+)$ is chamber-transitive on $\mathcal A^+$; moreover every chamber of $\mathcal A^+$ has finite stabiliser in $\mathcal K_A(F)$.
\end{proposition}

\begin{proof}
Let $C= (C^+,C^-)$ be the standard twin chamber and   $\mathcal{A}=(\mathcal{A}^+,\mathcal{A}^-)$ the standard twin apartment.

We first prove  that $\Stab_{\mathcal K_A(F)}(\mathcal A^+)$ is transitive on the chambers of $\mathcal A^+$. Indeed, for any $i \in I$, consider the rank~$1$ subgroup $G_i \cong \mathrm{SL}_2(F)$. By the definition of the Chevalley involution $\omega$, we know that $G_i$ is $\omega$-invariant and that $\omega$ acts on $G_i$ as the transpose-inverse automorphism. Therefore we have $\mathcal K_A(F)\cap G_i = \psi_i(\mathrm{SO}_2(F))$. In particular we have $r_i = \psi_i( \left(
\begin{array}{cc}
0 & 1
\\
-1 & 0
\end{array}
\right)
) \in \mathcal K_A(F)$. 
By construction of the building $X^+$, the element $r_i$ stabilises $\mathcal A^+$ and maps the chamber $C^+$ to the unique chamber of $\mathcal A^+$ which is $r_i$-adjacent to it. Therefore the subgroup $\tilde W = \fgen{ r_1, \dots, r_n } \leq \mathcal K_A(F)$ stabilises $\mathcal A^+$ and is transitive on the set of its chambers. This proves the claim.

It remains to prove that the stabiliser in $\mathcal K_A(F)$ of every chamber of $\mathcal A^+$ is finite. By the claim, it suffices to prove this for $C^+$. For each $g \in  \Stab_{\mathcal K_A(F)}(C^+)$, we have $g(C^-) = g\omega(C^+) = \omega g(C^+) = \omega(C^+) = C^-$. Therefore $\Stab_{\mathcal K_A(F)}(C^+) = \Stab_{\mathcal K_A(F)}(C^+, C^-) =   \mathcal K_A(F) \cap T$. By construction $T \cong (F^\times)^n$ and $\omega$ preserves $T$ and acts on it by taking inverses. Therefore an element $\eta(y_1, \dots, y_n) \in T$ is fixed by $\omega$ if and only if $y_i^2 =1$ for all $i$. The latter polynomial equation has at most two roots in $F$, so that $ |\mathcal K_A(F) \cap T| \leq 2^n$. 
\end{proof}

To complete the proof of the Main Theorem and of Corollary~\ref{cor:sharper}, we apply Theorem~\ref{thm:AcylBuilding} to deduce that the orthogonal form $\mathcal K(F)$, and more generally any subgroup containing a finite index subgroup of $\Stab_{\mathcal K_A(F)}(\mathcal A^+)$,  is acylindrically hyperbolic.

\subsection{Graph products}\label{sec:GraphProd}

Let $\Gamma = (V, E)$ be a finite simple graph with more than one vertex. Let $(G_v)_{v \in V}$ be a collection of non-trivial groups indexed by $V$. The \textbf{graph product} of the groups $(G_v)_{v \in V}$ along $\Gamma$ is the group $G$ defined as the quotient of the free product $\ast_{v \in V} G_v$ by the relations $[G_v, G_w]=1$, where $\set{v, w}$ runs over all pairs of vertices spanning an edge of $\Gamma$.  We say that the graph product is \textbf{irreducible} if $\Gamma$ is not a join, i.e. if $V$ has no non-trivial partition $V = V_1 \cup V_2$ such that every vertex in $V_1$ is adjacent to every vertex in $V_2$. Notice that if $\Gamma$ is not irreducible, then the graph product $G$ splits non-trivially as a direct product of groups, and thus cannot be acylindrically hyperbolic (unless one of the factors is finite). 

Define $W_\Gamma$ to be the right-angled Coxeter group indexed by $\Gamma$, i.e. the graph product of groups of order~$2$ along $\Gamma$. An important feature of graph products along $\Gamma$ is that they act on right-angled buildings with Weyl group $W_\Gamma$. 

\begin{proposition}\label{prop:GraphProd}
The graph product $G$ of the groups $(G_v)_{v \in V}$  along the graph $\Gamma$ acts by automorphisms on a right-angled building $X$ with Weyl group $W_\Gamma$. The action on the set of chambers is free and transitive. 
\end{proposition}
\begin{proof}
See \cite[Theorem~5.1]{Davis}.
\end{proof}

\begin{proof}[Proof of the Corollary]
If $\Gamma$ has exactly two vertices, then $G$ is a non-trivial free product. Therefore the $G$-action on the associated Bass--Serre tree is acylindrical, and the desired result follows. 

If $\Gamma$ has more than two vertices, then the building $X$ afforded by Proposition~\ref{prop:GraphProd} is of irreducible, non-spherical and non-affine type. Proposition~\ref{prop:GraphProd} therefore ensures that the hypotheses of Theorem~\ref{thm:AcylBuilding} are satisfied.
\end{proof}

\end{document}